\documentclass[12pt,reqno]{amsart}
\usepackage{latexsym}
\usepackage{amssymb}
\usepackage{verbatim}
\usepackage{version}

\catcode`\@=11
\def\sideset#1#2#3{%
  \@mathmeasure\z@\displaystyle{#3}%
  \global\setbox\@ne\vbox to\ht\z@{}\dp\@ne\dp\z@
  \setbox\tw@\box\@ne
  \@mathmeasure4\displaystyle{\copy\tw@#1}%
  \@mathmeasure6\displaystyle{#3{#2}}%
  \dimen@-\wd6 \advance\dimen@\wd4 \advance\dimen@\wd\z@
  \hbox to\dimen@{}\mathop{\kern-\dimen@\box4\box6}%
}
\catcode`\@=12

\input texdraw

\catcode`\@=11
\@namedef{subjclassname@2010}{%
  \textup{2010} Mathematics Subject Classification}
\catcode`\@=12

\newtheorem{theorem}{Theorem}
\newtheorem{proposition}[theorem]{Proposition}

\newtheorem{lemma}[theorem]{Lemma}
\newtheorem{corollary}[theorem]{Corollary}

\theoremstyle{remark}
\newtheorem{example}[theorem]{Example}

\newtheorem{remark}[theorem]{Remark}

\numberwithin{equation}{section}

\def\coef#1{\left\langle#1\right\rangle}
\def\fl#1{\left\lfloor#1\right\rfloor}
\def\lcm{\operatorname{lcm}}

\def\PSL{\operatorname{PSL}}
\newcommand{\Z}{\mathbb{Z}}

\newcommand{\R}{\mathbb{\R}}

\def\al{\alpha}
\def\be{\beta}
\def\ga{\gamma}

\includeversion{comment}

\begin{document}

\title[Free subgroup numbers modulo prime powers]{Free 
subgroup numbers modulo prime powers: the non-periodic case}  
\author[C. Krattenthaler and 
T.\,W. M\"uller]{C. Krattenthaler$^{\dagger}$ and
T. W. M\"uller$^*$} 

\address{$^{\dagger}$Fakult\"at f\"ur Mathematik, Universit\"at Wien, 
Oskar-Morgenstern-Platz~1, A-1090 Vienna, Austria.
WWW: {\tt http://www.mat.univie.ac.at/\lower0.5ex\hbox{\~{}}kratt}.}

\address{$^*$School of Mathematical Sciences, Queen Mary
\& Westfield College, University of London,
Mile End Road, London E1 4NS, United Kingdom.}

\thanks{$^\dagger$Research partially supported by the Austrian
Science Foundation FWF, grants Z130-N13 and S50-N15,
the latter in the framework of the Special Research Program
``Algorithmic and Enumerative Combinatorics"\newline\indent
$^*$Research supported by Lise Meitner Grant M1661-N25 of the Austrian
Science Foundation FWF}



\begin{abstract}
In [{\it J. Algebra} {\bf452} (2016), 372--389],
we characterise when the sequence of free subgroup numbers of a
finitely generated virtually free group $\Gamma$ 
is ultimately periodic modulo a given prime power. 
Here, we show that, in the remaining cases, in which the sequence of
free subgroup numbers is not ultimately periodic modulo a given
prime power, the number of free subgroups of index~$\lambda$ 
in $\Gamma$ is --- essentially --- 
congruent to a binomial coefficient times a rational function 
in $\lambda$ modulo a power of a prime that divides a certain invariant
of the group $\Gamma$, 
respectively to a binomial sum involving such numbers.
These results, apart from their intrinsic interest,
in particular allow for a much more efficient computation of
congruences for free subgroup numbers in these cases compared
to the direct recursive computation of these numbers implied
by the generating function results in [{\it J. London Math. Soc.}~(2)
\textbf{44} (1991), 75--94].
\end{abstract}

\maketitle

\section{Introduction}
\label{sec:intro}
For a finitely generated virtually free group $\Gamma$, denote by
$m_\Gamma$ the least common multiple of the orders of the finite
subgroups in $\Gamma$. Moreover, for a positive integer $\lambda$, let
$f_\lambda(\Gamma)$ be the number of free subgroups of index $\lambda
m_\Gamma$ in $\Gamma$. In \cite{KMPeriod}, a complete characterisation
is given of all pairs $(\Gamma, p^\alpha)$, where $\Gamma$ is a
finitely generated virtually free group and $p^\alpha$ is a proper
prime power, for which the sequence
$(f_\lambda(\Gamma))_{\lambda\geq1}$ is ultimately periodic modulo
$p^\alpha$. As it turns out, somewhat surprisingly, this is always the
case, unless $\mu_p(\Gamma) = 0$ and $\mu(\Gamma)\ge2$,\footnote{The
  condition $\mu(\Gamma)\ge2$ is equivalent to the assertion that
$\Gamma$ contains a non-Abelian free subgroup.} where $\mu_p(\Gamma)$
and $\mu(\Gamma)$ are certain invariants of
$\Gamma$  defined in Section~\ref{sec:back}. See also \cite{KM2},
where more precise results are obtained for lifts of the inhomogeneous
modular group.
Our present paper focuses on the latter case
of non-periodic behaviour. 
It is shown in \cite{MuPuFree} that, for $\mu_p(\Gamma)=0$ and
$\mu(\Gamma)\ge2$, the function $f_\lambda(\Gamma)$ satisfies the
congruence  
\begin{equation}
\label{Eq:fModp}
f_\lambda(\Gamma) \equiv (-1)^{\frac{(\mu(\Gamma)-1)\lambda+1}{p-1}}
\frac 1\lambda
\binom{\frac{\mu(\Gamma)\lambda}{p-1}}{\frac{\lambda-1}{p-1}}
\quad (\text{mod } p),
\end{equation}
where the binomial coefficient is defined to be zero whenever the
lower argument is not an integer;
cf.\ \cite[Eqn.~(35)]{MuPuFree}. 

The purpose of the present paper is to demonstrate that, under the same
assumptions, the function $f_\lambda(\Gamma)$ satisfies a
very similar congruence modulo an {\it arbitrary} $p$-power.
More precisely, if $\mu(\Gamma) \equiv 0,
1\pmod{p}$, the function $f_\lambda(\Gamma)$, when reduced modulo
any fixed $p$-power, is congruent to
a (quasi-)rational factor in $\lambda$ times a binomial
coefficient (see Corollary~\ref{cor:1} in Section~\ref{sec:main}), 
while in the remaining cases the right-hand side takes
the form of a sum of such expressions (see Corollary~\ref{cor:2} in
the same section). A remarkable consequence of these results is 
that, while it may be
safely conjectured that the generating function for the free subgroup 
numbers (which satisfies a highly non-linear differential equation
obtained from \eqref{Eq:G(z)Diff} via \eqref{Eq:GFTransform}) 
is not D-finite, implying that the sequence of free subgroup numbers
itself is not P-recursive,\footnote{The reader is referred to 
\cite[Ch.~6]{StanBI} for information
on D-finite series and P-recursive sequences.} 
its reduction modulo any fixed $p$-power is.

While the result obtained in Corollary~\ref{cor:1} for arbitrary $p$-powers
is `as good as' the mod~$p$ result \eqref{Eq:fModp},
we show in Proposition~\ref{prop:rek} 
that the sum described in Corollary~\ref{cor:2} 
satisfies an inhomogeneous linear recurrence 
of finite depth with 
constant coefficients and leading coefficient~1 
(which may be found automatically by means
of the Gosper--Zeilberger algorithm; cf.\ \cite{PeWZAA}). 
This leads again to an
efficient computation of $f_\lambda(\Gamma)$ modulo $p^\alpha$. 
All these results are presented in Section~\ref{sec:main}, and are illustrated
there by concrete examples.

 The only known earlier results concerning congruences of
free subgroup numbers modulo prime powers in the non-periodic case
covered the following scenarios: (i) lifts 
of Hecke groups $\mathfrak{H}(q) \cong C_2 \ast C_q$ with $q$ a Fermat
prime and $p=2$, and (ii) lifts of the classical modular group
$\mathfrak{H}(3)\cong \PSL_2(\Z)$ 
and $p=3$;  
see Section~8, Corollary~34 and Theorem~35 in
\cite{KKM}, and \cite[Sec.~16]{KM}. In particular, the behaviour of
$f_\lambda(\Gamma)$ modulo $p$-powers in these known cases fits into
the framework of the semi-automatic method for obtaining congruences
developed in \cite{KKM,KM,KM3}, 
and further (unpublished) work. As we show in this paper,
for finitely generated virtually free
groups $\Gamma$ and primes $p$ with $\mu_p(\Gamma) = 0$ and
$\mu(\Gamma)\ge2$, this is always the case. 

This semi-automatic method is based on a generating
function approach, featuring a basic series --- to be adapted
for each class of applications --- which is then used to express
the generating function for the sequence of numbers we have in mind,
reduced modulo a given $p$-power, as a polynomial in this
basic series. 
We show in Theorem~\ref{thm:1} in Section~\ref{sec:GF} 
that, if $\mu_p(\Gamma)=0$
and $\mu(\Gamma)\ge2$, we may choose the series
\begin{equation} \label{eq:Phi}
\Phi(z)=\sum_{n=1}^\infty
(-1)^{\frac{(\mu(\Gamma)-1)n+1}{p-1}}
\frac {1} {n}
\binom{\frac{\mu(\Gamma)n}{p-1}}{\frac{n-1}{p-1}}\,z^n
\end{equation}
as basic series 
(i.e., the series formed out of the coefficients
on the right-hand side of \eqref{Eq:fModp}) 
in order to express the generating function
$\sum_{\lambda=1}^\infty f_\lambda(\Gamma)\,z^\lambda$ 
modulo $p$-powers as a polynomial
in~$\Phi(z)$. 
Corollaries~\ref{cor:1} and \ref{cor:2} alluded to above are
consequences of Theorem~\ref{thm:1}.
The proof of the theorem requires some auxiliary
results which are presented in Section~\ref{sec:aux}.
These include some interesting determinant evaluations, see
Lemmas~\ref{lem:M} and \ref{lem:detA}.

A remarkable feature of the present application of our semi-automatic
generating function method is that the degree of the polynomial
in the basic series $\Phi(z)$ expressing the generating function
$\sum_{\lambda=1}^\infty f_\lambda(\Gamma)\,z^\lambda$ modulo~$p^\alpha$
does not increase with~$\alpha$. As a consequence, the complexity
of the computation only mildly increases with~$\alpha$.
This is in sharp contrast to our previous applications of this method. 
The reason for the above phenomenon lies in the fact that $\Phi(z)$
satisfies an {\it exact\/} functional equation over the integers, 
namely \eqref{eq:PhiGl}, while in our previous applications the
basic series satisfied functional equations 
modulo~$p^\alpha$ of complexity increasing with~$\alpha$.

Our
results concerning the function $f_\lambda(\Gamma)$ are complemented
by Theorem~\ref{thm:mup=0} in Section~\ref{sec:char}, 
which precisely characterises those
finitely generated virtually free groups $\Gamma$ with
$\mu(\Gamma)\ge2$ and $\mu_p(\Gamma)=0$ for a given prime number $p$. 

\section{Some preliminaries on finitely generated virtually free groups}
\label{sec:back}

Our notation and terminology concerning virtually free
groups and their decomposition in terms of a graph of groups follows
Serre's book \cite{Serre2}; 
in particular, the category of graphs used in the context of graphs
of groups is described in
\cite[\S2]{Serre2}. 
This category deviates slightly from the usual notions in graph
theory. In order to distinguish the objects of this category from
graphs in the sense of graph theory, we call them {\it $S$-graphs}.
Specifically, an {\it $S$-graph} $X$ consists of two sets:
$E(X)$, the set of (directed) {\it edges}, and $V(X)$, the set of
{\it vertices}. The set $E(X)$ is endowed with a fixed-point-free involution
${}^-: E(X) \rightarrow E(X)$ ({\it reversal of orientation}), and there are
two functions $o,t: E(X)\rightarrow V(X)$ assigning to an edge $e\in
E(X)$ its {\it origin} $o(e)$ and {\it terminus} $t(e)$, such that
$t(\bar{e}) = 
o(e)$. 
The reader should note that, according to the above
definition, $S$-graphs may have loops (that is, edges $e$ with
$o(e)=t(e)$) and multiple edges (that is, several edges with
the same origin and the same terminus).  
An {\it orientation} $\mathcal{O}(X)$ consists of a choice of exactly 
one edge in each pair $\{e, \bar{e}\}$ 
(this is indeed always a pair -- even for loops --
since, by definition, the involution
${}^-$ is fixed-point-free). Such a pair is called a
{\it geometric edge}. 

\medskip
Let $\Gamma$ be a finitely generated virtually free group with
Stallings decomposition\break
 $(\Gamma(-), X)$; that is, $(\Gamma(-), X)$ is
a finite graph of finite groups with fundamental group 
$\pi_1(\Gamma(-), X) \cong \Gamma$. Replacing the stabiliser groups
of vertices and edges by their respective group orders and 
replacing each pair $(e,\bar e)$ by one unoriented edge, we obtain
the corresponding {\it order graph} of~$\Gamma$. Abstractly, an {\it order graph} is
a finite connected unoriented graph (in the sense of graph theory;
multiple edges and loops are allowed) 
whose vertices~$v$ and edges~$e$ carry positive
integers, $n(v)$, respectively $n(e)$, as labels 
such that $n(e)\mid n(v)$ if $v$ is incident
to~$e$. The labels of vertices and edges will frequently be referred
to as their respective {\it order}.
 
As in the introduction, denote by $m_\Gamma$ the least common multiple
of the orders of the finite subgroups in $\Gamma$, so that, in
terms of the above Stallings decomposition of $\Gamma$,  
\begin{equation} \label{eq:mG} 
m_\Gamma = \mathrm{lcm}\big\{\vert\Gamma(v)\vert:\, v\in V(X)\big\}.
\end{equation}
(This formula essentially follows from the well-known fact that a
finite group has a fixed point when acting on a tree.)  
The \textit{type} $\tau(\Gamma)$ of a finitely generated virtually free group
$\Gamma \cong \pi_1(\Gamma(-), X)$ is defined as the  
tuple
$$
\tau(\Gamma) = \big(m_\Gamma; \zeta_1(\Gamma), \ldots,
\zeta_\kappa(\Gamma), \ldots, \zeta_{m_\Gamma}(\Gamma)\big), 
$$
where the  
$\zeta_\kappa(\Gamma)$'s are integers indexed by the divisors $\kappa$ of
$m_\Gamma$, given by 
\begin{equation} \label{eq:zeta} 
\zeta_\kappa(\Gamma) = 
\big\vert\big\{e\in \mathcal{O}(X):\, \vert\Gamma(e)\vert
\,\big\vert\, \kappa\big\}\big\vert\,-\, \big\vert\big\{v\in V(X):\,
\vert\Gamma(v)\vert \,\big\vert\, \kappa\big\}\big\vert. 
\end{equation}
(Here, $\mathcal{O}(X)$ is any orientation of the $S$-graph $X$.) It can
be shown that the type $\tau(\Gamma)$ is in fact an invariant 
of the group $\Gamma$, i.e., independent of the particular
decomposition of $\Gamma$ in terms of a graph of groups $(\Gamma(-),
X)$, and that two finitely generated virtually free groups $\Gamma_1$
and $\Gamma_2$ contain the same number of free subgroups of index $n$
for each positive integer $n$ if, and only if, $\tau(\Gamma_1) =
\tau(\Gamma_2)$; cf.\ \cite[Theorem~2]{MuDiss}. We have
$\zeta_\kappa(\Gamma)\geq0$ for $\kappa<m_\Gamma$ and
$\zeta_{m_\Gamma}(\Gamma)\geq-1$ with equality occurring in the latter
inequality if, and only if, $\Gamma$ is the fundamental group of a
tree of groups; cf.\ \cite[Prop.~1]{MuEJC} or \cite[Lemma~2]{MuDiss}. 

Inspection of \eqref{eq:mG} and \eqref{eq:zeta} reveals that all
ingredients of the type (that is, $m_\Gamma$ and the $\zeta_\kappa$'s)
depend only on the orders but not on the internal structure of the
stabilisers of vertices and edges of $(\Gamma(-),X)$.
Therefore it makes sense to attach the same invariants to the
order graph obtained from the graph of groups $(\Gamma(-),X)$
in the way described earlier, or, more generally, to an abstract
order graph. 
Specifically, given an order graph~$G$, we define
$m_G$ to be the least common multiple of the vertex orders $n(v)$,
taken over all vertices of~$G$, and, for a divisor $\kappa$ 
of $m_G$, we let 
\begin{equation} \label{eq:zetaG} 
\zeta_\kappa(G) = 
\big\vert\big\{e\in E(G):\,n(e)
\mid \kappa\big\}\big\vert\,-\, \big\vert\big\{v\in V(G):\,
n(v) \mid \kappa\big\}\big\vert,
\end{equation}
where $V(G)$ denotes the set of vertices of $G$ and $E(G)$ the
set of edges. 

\medskip
Define a
\textit{torsion-free $\Gamma$-action} on a set $\Omega$ to be a
$\Gamma$-action on $\Omega$ which is free when restricted to finite
subgroups, and let 
\begin{equation} \label{eq:gla} 
g_\lambda(\Gamma):= \frac{\mbox{number of torsion-free
$\Gamma$-actions on a set with $\lambda m_\Gamma$ elements}}{(\lambda
m_\Gamma)!},\quad \lambda\geq0; 
\end{equation}
in particular, $g_0(\Gamma)=1$. The sequences
$\big(f_\lambda(\Gamma)\big)_{\lambda\geq1}$ and
$\big(g_\lambda(\Gamma)\big)_{\lambda\geq0}$ are related via the
Hall-type transformation formula\footnote{See \cite[Cor.~1]{MuDiss}, or
\cite[Prop.~1]{DM} for a more general result.}  
\begin{equation}
\label{Eq:Transform}
\sum_{\mu=0}^{\lambda-1} g_\mu(\Gamma) f_{\lambda-\mu}(\Gamma) =
m_\Gamma \lambda g_\lambda(\Gamma),\quad \lambda\geq1. 
\end{equation}
Introducing the generating functions 
\[
F_\Gamma(z) := \sum_{\lambda=1}^\infty f_{\lambda}(\Gamma) z^\lambda
\,\mbox{ and }\, G_\Gamma(z):= \sum_{\lambda=0}^\infty g_\lambda(\Gamma)
z^\lambda, 
\]
Equation~\eqref{Eq:Transform} is seen to be equivalent to the relation
\begin{equation}
\label{Eq:GFTransform}
F_\Gamma(z) = m_\Gamma z\frac{d}{dz}\big(\log G_\Gamma(z)\big). 
\end{equation}

\medskip
Define the \textit{free rank} $\mu(\Gamma)$ of a finitely generated
virtually free group $\Gamma$ to be the rank of a free subgroup of
index $m_\Gamma$ in $\Gamma$ (existence of such a subgroup follows,
for instance, from Lemmas~8 and 10 in \cite{Serre2}; it need not be 
unique, though). It can be shown that the free rank $\mu(\Gamma)$ may be
expressed in terms of the type of $\Gamma$ via  
\begin{equation}
\label{Eq:muTypeRewrite}
\mu(\Gamma) = 1 + \sum_{\kappa\mid m_\Gamma} \varphi(m_\Gamma/\kappa)
\zeta_\kappa(\Gamma),
\end{equation}
which shows in particular that $\mu(\Gamma)$ is well-defined.
It is known that the sequence
$g_\lambda(\Gamma)$ is of hypergeometric type and that its generating
function $G_\Gamma(z)$ satisfies a homogeneous linear differential
equation 
\begin{equation}
\label{Eq:G(z)Diff}
\theta_0(\Gamma) G_\Gamma(z) \,+ \,(\theta_1(\Gamma) z - m_\Gamma)
G'_\Gamma(z) \,+ \,\sum _{\mu=2}^{\mu(\Gamma)} \theta_\mu(\Gamma)
z^\mu G^{(\mu)}_\Gamma(z) = 0 
\end{equation}
of order $\mu(\Gamma)$ with integral coefficients $\theta_\mu(\Gamma)$
given by 
\begin{equation}
\label{Eq:G(z)DiffCoeffs}
\theta_\mu(\Gamma) = \frac{1}{\mu!} \sum_{j=0}^\mu (-1)^{\mu-j}
\binom{\mu}{j} m_\Gamma (j+1) \prod_{\kappa\mid
  m_\Gamma}\,\underset{(m_\Gamma, k) = \kappa}{\prod_{1\leq k\leq
    m_\Gamma}} (jm_\Gamma + k)^{\zeta_\kappa(\Gamma)},\quad 0\leq
\mu\leq \mu(\Gamma); 
\end{equation}
cf.\ \cite[Prop.~5]{MuDiss}.

For a finitely generated virtually free group $\Gamma$ and a prime
number $p$, we introduce, in formal analogy with formula
\eqref{Eq:muTypeRewrite}, the  
\textit{$p$-rank $\mu_p(\Gamma)$} of $\Gamma$ via the equation 
\begin{equation}
\label{Eq:mupDef}
\mu_p(\Gamma) = 1 + \sum_{p\mid \kappa \mid m_\Gamma}
\varphi(m_\Gamma/\kappa) \zeta_\kappa(\Gamma). 
\end{equation}
Clearly, $\mu_p(\Gamma)\geq0$, with equality occurring if, and only if, $\Gamma$ is the fundamental group of a
tree of groups, $p \mid m_\Gamma$, and $\zeta_\kappa(\Gamma)=0$ for
$p\mid \kappa\mid 
m_\Gamma$ and $\kappa<m_\Gamma$. 
Since the free rank $\mu(\Gamma)$ and the $p$-rank $\mu_p(\Gamma)$ only depend
on the type invariants $m_\Gamma$ and the $\zeta_\kappa$'s, in view
of our earlier discussion they may also be defined for abstract
order graphs and, in particular, for an order graph $G$ of a finitely
generated virtually free group~$\Gamma$. Doing so, one has
$\mu(\Gamma)=\mu(G)$ and $\mu_p(\Gamma)=\mu_p(G)$. These conventions
will be used in the proof of Theorem~\ref{thm:mup=0}.

\medskip

In what follows, it will be important to be able to represent a
finitely generated 
virtually free group $\Gamma$ by a graph of groups avoiding trivial
amalgamations along a maximal tree. This is achieved via the following
auxiliary result.

\begin{lemma}[\sc Normalisation]
\label{Lem:Normalise}
Let $(\Gamma(-), X)$ be a {\em(}connected{\em)} graph of groups with
fundamental 
group $\Gamma,$ and suppose that $X$ has only finitely many
vertices. Then there exists a graph of groups $(\Delta(-), Y)$ with
$\vert V(Y)\vert < \infty$ and a spanning tree $T$ in $Y,$ such that
$\pi_1(\Delta(-),Y) \cong \Gamma$, and such that\footnote{The notation 
used in Equation~\eqref{eq:normal} follows Serre; 
see D\'ef.~8 in \cite[Sec.~4.4]{Serre2}.} 
\begin{equation} \label{eq:normal}
\Delta(e)^e \neq \Delta(t(e))\,\mbox{ and }\, \Delta(e)^{\bar{e}} \neq
\Delta(o(e)),\quad 
\text {for }e\in E(T). 
\end{equation}
Moreover, if $(\Gamma(-), X)$ satisfies the finiteness condition
\bigskip

\centerline{\hskip3cm
$(F_1)$ \qquad \mbox{$X$ is a finite $S$-graph,}\hfill}
\bigskip

\noindent
or
\bigskip

\centerline{\hskip3cm $(F_2)$  \qquad 
\mbox{$\Gamma(v)$ is finite for every vertex $v\in V(X),$}\hfill}
\bigskip

\noindent
then we may choose $(\Delta(-), Y)$ so as to enjoy the same property.
\end{lemma}
See \cite[Sec.~3]{KMNormal} for a proof of this useful result.
Subsequently, we shall call a graph of groups $(\Delta(-), Y)$ {\it
  normalised}, if it
satisfies the conditions of the lemma for some
spanning tree $T$ of $Y$. In our situation, normalised graphs of groups
will always be trees, so coincide with their respective spanning
trees. We shall therefore suppress the reference to the spanning trees
from now on.

\section{Characterisation of finitely generated virtually free groups
$\Gamma$  with $\mu_p(\Gamma)=0$}
\label{sec:char}

Recall (see paragraph below \eqref{Eq:mupDef}) that, if a finitely
generated virtually free group satisfies $\mu_p(\Gamma)=0$ for a given
prime~$p$, then, in particular, $\Gamma$ is the fundamental group of a
tree of groups.
Theorem~\ref{thm:mup=0} below tells us how a normalised
(in the sense of Lemma~\ref{Lem:Normalise}) order tree\footnote{Here,
  {\it order tree} means an order graph which has the form of a tree.} $X$
underlying the Stallings decomposition $(\Gamma(-),X)$ 
of a finitely generated virtually
free group $\Gamma$ must be constructed so as to satisfy $\mu_p(\Gamma)=0$. 

Given a fixed prime number~$p$,
the starting point of our construction are certain finite rooted
vertex-labelled trees which we call {\it divisor trees}.
By definition, 
vertices of divisor trees
are labelled by positive integers coprime to~$p$. Moreover,
any two
adjacent vertices, say $v_1$ and $v_2$, with $v_1$ closer to the root
than $v_2$, satisfy $\ell(v_2)\mid \ell(v_1)$ and 
$\ell(v_2)< \ell(v_1)$, where $\ell(v_1)$ and $\ell(v_2)$ denote the
labels of $v_1$ and $v_2$. 
See Figure~\ref{fig:3} for an example of such a divisor tree.
There, the prime number to be fixed from the very beginning is $p=5$.
In the figure, the root is indicated by a square.

\begin{figure}
\unitlength.4cm
\begin{picture}(15,11)(-3,-1)
\put(4,5){\line(1,0){3}}
\put(4,5){\line(-1,1){3}}
\put(1,8){\line(-1,0){3}}
\put(4,5){\line(-1,-1){3}}
\put(7,5){\line(1,1){3}}
\put(7,5){\line(1,-1){3}}
\put(4,5){\raise0pt\hbox{\kern-4pt\large \vrule width7pt height3.5pt depth3.5pt}}
\put(-2,8){\raise0pt\hbox{\kern-3pt\circle*{.5}}}
\put(1,8){\raise0pt\hbox{\kern-3pt\circle*{.5}}}
\put(1,2){\raise0pt\hbox{\kern-3pt\circle*{.5}}}
\put(7,5){\raise0pt\hbox{\kern-3pt\circle*{.5}}}
\put(10,8){\raise0pt\hbox{\kern-3pt\circle*{.5}}}
\put(10,2){\raise0pt\hbox{\kern-3pt\circle*{.5}}}
\put(10.7,8){\raise-5pt\hbox{\kern-5pt 2}}
\put(10.7,2){\raise-5pt\hbox{\kern-5pt 3}}
\put(7,6){\raise-5pt\hbox{\kern-5pt 12}}
\put(4.4,6.1){\raise-5pt\hbox{\kern-5pt 504}}
\put(0,8){\raise5pt\hbox{\kern-1pt 42}}
\put(-3,8){\raise5pt\hbox{\kern-1pt 6}}
\put(0,2){\raise-5pt\hbox{\kern-5pt 2}}
\end{picture}
\vskip-.5cm
\caption{A divisor tree for $p=5$}
\label{fig:4}
\end{figure}
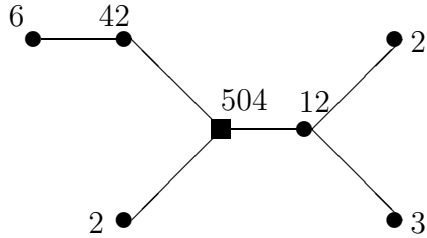

Given a divisor tree $D$, we fix a map $f$ from the set of vertices of $D$
into the set of (ordinary, unlabelled) finite 
rooted trees with the property that non-root vertices are mapped
to non-trivial\footnote{Here, `non-trivial' means `at least two vertices'.} 
trees rooted at a leaf and, 
in case $D$ consists of just the root, this root must be
mapped to a non-trivial tree (with no restriction on the location of
the root). Figure~\ref{fig:6} shows an example of such a map~$f$
defined on the vertices of the order tree in Figure~\ref{fig:4}.
There, the roots of the image trees are indicated by little circles.

From $D$ and $f$, we construct a certain class of order graphs.
The reader is advised to consult Figure~\ref{fig:5}
while reading the description of this construction in the following
paragraph. We remind the reader that, in that example, the fixed 
prime number is $p=5$.

\begin{figure}
\unitlength.4cm
\begin{picture}(15,21)(-3,-9)
\put(7,8){\line(1,0){5.5}}
\put(7,5){\line(1,0){2.5}}
\put(7,2){\line(1,0){5.5}}
\put(9.5,2){\line(-2,1){2}}
\put(9.5,2){\line(-2,-1){2}}
\put(7,-1){\line(2,1){2.5}}
\put(7,-1){\line(2,-1){2.5}}
\put(7,-4){\line(1,0){5.5}}
\put(7,-7){\line(1,0){2.5}}
\put(1,11){\raise0pt\hbox{\kern-4pt\large \vrule width7pt height3.5pt depth3.5pt}}
\put(1,8){\raise0pt\hbox{\kern-3pt\circle*{.5}}}
\put(1,5){\raise0pt\hbox{\kern-3pt\circle*{.5}}}
\put(1,2){\raise0pt\hbox{\kern-3pt\circle*{.5}}}
\put(1,-1){\raise0pt\hbox{\kern-3pt\circle*{.5}}}
\put(1,-4){\raise0pt\hbox{\kern-3pt\circle*{.5}}}
\put(1,-7){\raise0pt\hbox{\kern-3pt\circle*{.5}}}
\put(3,11){\raise-3pt\hbox{\kern-3pt$\longrightarrow$}}
\put(3,8){\raise-3pt\hbox{\kern-3pt$\longrightarrow$}}
\put(3,5){\raise-3pt\hbox{\kern-3pt$\longrightarrow$}}
\put(3,2){\raise-3pt\hbox{\kern-3pt$\longrightarrow$}}
\put(3,-1){\raise-3pt\hbox{\kern-3pt$\longrightarrow$}}
\put(3,-4){\raise-3pt\hbox{\kern-3pt$\longrightarrow$}}
\put(3,-7){\raise-3pt\hbox{\kern-3pt$\longrightarrow$}}
\put(0,11){\raise-5pt\hbox{\kern-15pt 504}}
\put(0,8){\raise-5pt\hbox{\kern-5pt 6}}
\put(0,-1){\raise-5pt\hbox{\kern-10pt 12}}
\put(0,2){\raise-5pt\hbox{\kern-5pt 2}}
\put(-2,2){\raise-15pt\hbox{\kern-10pt (left-bottom)}}
\put(0,5){\raise-5pt\hbox{\kern-10pt 42}}
\put(0,-4){\raise-5pt\hbox{\kern-5pt 2}}
\put(-2,-4){\raise-15pt\hbox{\kern-10pt (right-top)}}
\put(0,-7){\raise-5pt\hbox{\kern-5pt 3}}
\put(7,11){\raise0pt\hbox{\kern-3pt\circle{.5}}}
\put(7,8){\raise0pt\hbox{\kern-3pt\circle*{.5}}}
\put(10,8){\raise0pt\hbox{\kern-3pt\circle*{.5}}}
\put(13,8){\raise0pt\hbox{\kern-3pt\circle{.5}}}
\put(7,-4){\raise0pt\hbox{\kern-3pt\circle{.5}}}
\put(10,-4){\raise0pt\hbox{\kern-3pt\circle*{.5}}}
\put(13,-4){\raise0pt\hbox{\kern-3pt\circle*{.5}}}
\put(7,-1){\raise0pt\hbox{\kern-3pt\circle{.5}}}
\put(10,0.3){\raise0pt\hbox{\kern-3pt\circle*{.5}}}
\put(10,-2.3){\raise0pt\hbox{\kern-3pt\circle*{.5}}}
\put(7,5){\raise0pt\hbox{\kern-3pt\circle*{.5}}}
\put(10,5){\raise0pt\hbox{\kern-3pt\circle{.5}}}
\put(7,2){\raise0pt\hbox{\kern-3pt\circle*{.5}}}
\put(10,2){\raise0pt\hbox{\kern-3pt\circle*{.5}}}
\put(13,2){\raise0pt\hbox{\kern-3pt\circle{.5}}}
\put(8,3){\raise0pt\hbox{\kern-3pt\circle*{.5}}}
\put(8,1){\raise0pt\hbox{\kern-3pt\circle*{.5}}}
\put(7,-7){\raise0pt\hbox{\kern-3pt\circle{.5}}}
\put(10,-7){\raise0pt\hbox{\kern-3pt\circle*{.5}}}
\end{picture}
\vskip-.5cm
\caption{A function $f$ on the vertices of the divisor tree of
  Figure~\ref{fig:4}}
\label{fig:6}
\end{figure}

\begin{figure}
\unitlength.4cm
\begin{picture}(15,15)(-3,-5)
\put(4,5){\line(1,0){3}}
\put(4,5){\line(-1,1){3}}
\put(1,8){\line(-1,0){6}}
\put(4,5){\line(-1,-1){3}}
\put(7,5){\line(1,1){3}}
\put(7,5){\line(1,-1){3}}
\put(10,8){\line(1,0){6}}
\put(10,2){\line(1,0){3}}
\put(1,2){\line(-1,0){4}}
\put(1,2){\line(-1,-3){1.5}}
\put(1,2){\line(2,-1){4}}
\put(4,5){\raise0pt\hbox{\kern-4pt\large \vrule width7pt height3.5pt depth3.5pt}}
\put(-5,8){\raise0pt\hbox{\kern-3pt\circle*{.5}}}
\put(-2,8){\raise0pt\hbox{\kern-3pt\circle*{.5}}}
\put(1,8){\raise0pt\hbox{\kern-3pt\circle*{.5}}}
\put(1,2){\raise0pt\hbox{\kern-3pt\circle*{.5}}}
\put(7,5){\raise0pt\hbox{\kern-3pt\circle*{.5}}}
\put(16,8){\raise0pt\hbox{\kern-3pt\circle*{.5}}}
\put(13,8){\raise0pt\hbox{\kern-3pt\circle*{.5}}}
\put(10,8){\raise0pt\hbox{\kern-3pt\circle*{.5}}}
\put(13,2){\raise0pt\hbox{\kern-3pt\circle*{.5}}}
\put(10,2){\raise0pt\hbox{\kern-3pt\circle*{.5}}}
\put(-3,2){\raise0pt\hbox{\kern-3pt\circle*{.5}}}
\put(-0.2,-2.4){\raise0pt\hbox{\kern-3pt\circle*{.5}}}
\put(5,0){\raise0pt\hbox{\kern-3pt\circle*{.5}}}
\put(16.7,8){\raise6pt\hbox{\kern-16pt 10}}
\put(13.7,8){\raise6pt\hbox{\kern-16pt 10}}
\put(10.7,8){\raise6pt\hbox{\kern-16pt 60}}
\put(13.7,2){\raise-13pt\hbox{\kern-16pt 15}}
\put(10.7,2){\raise-13pt\hbox{\kern-16pt 60}}
\put(7,6){\raise-5pt\hbox{\kern-5pt 60}}
\put(4.4,6.1){\raise-5pt\hbox{\kern-5pt 5040}}
\put(0,8){\raise5pt\hbox{\kern-1pt 210}}
\put(-2.8,8){\raise5pt\hbox{\kern-1pt 30}}
\put(-6,8){\raise5pt\hbox{\kern-1pt 30}}
\put(0,2){\raise5pt\hbox{\kern-0pt 10}}
\put(-4,2){\raise-5pt\hbox{\kern-10pt 10}}
\put(-1,-3){\raise-8pt\hbox{\kern-5pt 10}}
\put(5,0){\raise-15pt\hbox{\kern-5pt 10}}
\put(14.8,7){\raise0pt\hbox{\kern-10pt 2}}
\put(11.8,7){\raise0pt\hbox{\kern-10pt 2}}
\put(11.8,2){\raise3pt\hbox{\kern-10pt 3}}
\put(8.3,2.5){\raise3pt\hbox{\kern-10pt 12}}
\put(9.3,5.5){\raise3pt\hbox{\kern-10pt 12}}
\put(5.8,4){\raise0pt\hbox{\kern-10pt 12}}
\put(3.7,2.5){\raise3pt\hbox{\kern-10pt 2}}
\put(3.2,0){\raise3pt\hbox{\kern-10pt 2}}
\put(0.2,-1){\raise3pt\hbox{\kern-10pt 2}}
\put(-.8,.8){\raise3pt\hbox{\kern-10pt 2}}
\put(2.3,5.5){\raise3pt\hbox{\kern-10pt 42}}
\put(-.2,7){\raise0pt\hbox{\kern-10pt 6}}
\put(-3.2,7){\raise0pt\hbox{\kern-10pt 6}}
\end{picture}
\vskip-.5cm
\caption{An order tree resulting from the divisor tree of
  Figure~\ref{fig:4} and the function $f$ from Figure~\ref{fig:6}}
\label{fig:5}
\end{figure}

If $v_1$ and $v_2$ are adjacent in $D$, where
$v_1$ is closer to the root than $v_2$, then we glue the root of
$f(v_2)$ to one of the leaves of $f(v_1)$. If this is done for all
edges of $D$, we obtain a rooted tree, $U$ say, where the root of $U$
is by definition the root of $f(r)$, with $r$ being the root of $D$. 
In Figure~\ref{fig:5}, the root is again indicated by a square.
Given some vertex $v$ in $D$,
we label the edges in $f(v)$ by $\ell(v)$ and the non-root vertices
in $f(v)$ by $p\cdot\ell(v)$. The root of $U$ (that is, the root of $f(r)$) 
is assigned a number which is a multiple of
$$
p\cdot\lcm\{\ell(v):v\in D\}.
$$
Abusing notation, we write $f(D)$ for the set of order graphs
resulting from this construction. All of them are trees. We shall
occasionally use the term {\it order tree} for these order graphs.

\begin{theorem} \label{thm:mup=0}
Let $\Gamma\cong \pi_1(\Gamma(-),X)$ 
be a finitely generated virtually free group with $\mu(\Gamma)\ge2$, 
where $X$ is an $S$-graph
which is assumed to be normalised in the sense of 
Lemma~{\em\ref{Lem:Normalise}}.
Then $\mu_p(\Gamma)=0$ if, and only if,
there exist a divisor tree $D$ and a map
$f$ 
as above from the set of vertices of $D$ into the set of
finite rooted trees such that the order tree corresponding
to~$(\Gamma(-),X)$ is in $f(D)$.
\end{theorem}

\begin{proof}
We start with the proof of the forward implication.
We consider the order graph of $(\Gamma(-),X)$, which we denote by~$G$.
By the characterisation of groups~$\Gamma$ with
$\mu_p(\Gamma)=0$ given in the paragraph after \eqref{Eq:mupDef}, we 
know that $G$ is a tree, that, using the identification of invariants
of $\Gamma$ and $G$ discussed in Section~\ref{sec:back},
$p \mid m_G$, and that $\zeta_\kappa(G)=0$ for
$p\mid \kappa\mid m_G$ and $\kappa<m_G$. 
Also, since $(\Gamma(-), X)$ is normalised, we have $n(e) < n(v)$ if
$v\in V(G)$ is incident with $e\in E(G)$.  

Let $m$ be the minimal order of vertices and edges in $G$.
Since $(\Gamma(-),X)$ is assumed to be normalised and $\mu(\Gamma)\ge2$,
this order must be the order of an edge.
Furthermore, $m$ cannot be divisible by~$p$ since otherwise we would have
$$
0=\zeta_m(G)=
\big\vert\big\{e\in E(G):\, n(e)=m\big\}\big\vert>0,
$$
a contradiction. 

In order to proceed, we need to introduce an auxiliary object.
Let $\ell $ be a positive integer. We let $S_\ell (G)$ be the
collection of subtrees of $G$ consisting of those vertices and
edges with orders dividing~$\ell $.
It should be noted that the connected components of $S_\ell(G)$ 
need not be trees
in the classical sense since they may contain edges with one or both
of their vertices removed. 

Next we consider $S_{pm}(G)$. We claim that $S_{pm}(G)$
consists of all vertices and edges with orders $m$ or $pm$, but
no other vertices or edges. Let us assume for a contradiction
that there is a vertex~$v^*$ with $n(v^*)=pm'$ 
or an edge $e^*$ with $n(e^*)=pm'$, 
$m'\mid m$ and $m'<m$,
where $m'$ is minimal with this property. 
If there should be such a vertex~$v^*$, then it is
incident 
with an edge~$\tilde e$ with $n(\tilde e)$
properly dividing $n(v^*)=pm'$. The order 
$n(\tilde e)$ cannot be $m'$ since this would contradict
minimality of $m$. Thus, $n(\tilde e)=pm''$ with $m''\mid m'$
and $m''<m'$, contradicting the minimality of $m'$.
On the other hand, if there is an edge~$e^*$ as above,
then we have
$$
0=\zeta_{pm'}(G)=
\big\vert\big\{e\in E(G):\, n(e)
\mid pm'\big\}\big\vert\,-\, \big\vert\big\{v\in V(G):\,
n(v)\mid pm'\big\}\big\vert.
$$
Thus, there must be at least one vertex $v$ with
$n(v)\mid pm'$. If $n(v)< pm'$, then
we have again a contradiction to the minimality of~$m'$. If
$n(v)=pm'$, then the above argument for $v^*$ also
produces a contradiction to the minimality of~$m'$.

If $pm=m_G$, then $S_{pm}(G)=G$ and indeed
$\mu_p(\Gamma)=\mu_p(G)=0$. 

If $pm<m_G$, then the connected components of $S_{pm}(G)$
might be of two kinds: either an edge with order $pm$ without
vertices, or a subtree of $G$ consisting of edges with
order~$m$ and of some vertices of these edges, which
have order $pm$. If all vertices of the edges would be
part of the component, then this would already be the complete
tree $G$, which is impossible by our assumption that $pm<m_G$. We may
therefore assume that in each component there is at least one
vertex 
of some edge missing. This vertex must be a leaf of the tree
structure. (If not, the subtree would actually decompose into smaller
trees.) In that case, each component contributes
a non-negative number to
$$
\zeta_{pm}(G)=
\big\vert\big\{e\in E(G):\, n(e)
\mid pm\big\}\big\vert\,-\, \big\vert\big\{v\in V(G):\,
n(v)\mid pm\big\}\big\vert.
$$
Since $\zeta_{pm}(G)=0$, all components must actually contribute zero.
This implies that components of the first kind cannot exist, and all
components consist of edges of order~$m$ and an equal
number of vertices with order~$pm$, that is, exactly one
of the vertices is missing from the 
tree component. 

We now remove $S_{pm}(G)$ from $G$. What remains is another order
tree, say $G'$. It is easy to see that our construction guarantees
that $\zeta_\kappa(G')=\zeta_\kappa(G)$ for all $\kappa$.

We repeat the above construction for~$G'$, with
a new minimal order~$m'>m$. This process is continued until
nothing remains from the original order tree $G$. 

We now form a divisor tree out of the pieces of this construction.
Each connected component of $S_{pm}(G)$, of $S_{pm'}(G')$, \dots\
is interpreted as a vertex labelled by~$m$, by $m'$, \dots, respectively,
and two vertices, $v_1$ and $v_2$ say, are connected by an edge if 
the (incomplete) tree, $C(v_2)$ say, 
corresponding to $v_2$ was attached to the 
(incomplete) tree, $C(v_1)$ say, corresponding to $v_1$ in the original order
tree~$G$. The label of~$v_2$ divides the one of~$v_1$
since the order of the leaf of~$C(v_1)$ on which
$C(v_2)$ was attached must be a multiple of the order of the edges of
$C(v_2)$. 

\medskip
Finally, to see the reverse implication,
one has to convince oneself that the 
divisor tree construction of the theorem always yields order trees
$G$ with $\mu_p(G)=0$, and that we have $\mu_p(\Gamma)=0$ 
for any group~$\Gamma$ of the theorem with
order graph equal to~$G$, which is not difficult.
\end{proof}

\section{auxiliary results}
\label{sec:aux}

The purpose of this section is to provide the means for the proof
of Theorem~\ref{thm:1} in the next section. Lemma~\ref{lem:Phi'}
below demonstrates that the derivatives of our basic series
$\Phi(z)$ in \eqref{eq:Phi} can be expressed as a polynomial
in $\Phi(z)$ with rational coefficients, which is one of the
fundamental facts needed in the proof of Theorem~\ref{thm:1}.
The proof of the lemma is based on the evaluation of the
determinant of a block matrix given in Lemma~\ref{lem:M},
which itself uses another determinant evaluation, provided
in Lemma~\ref{lem:detA}. The determinant evaluation of
Lemma~\ref{lem:M} also plays a crucial role in the proof
of Theorem~\ref{thm:1}.

Let $p$ be a given prime number. In all of this section, we
write $N$ for $\mu(\Gamma)/(p-1)$. Using this notation, the series
$\Phi(z)$ in \eqref{eq:Phi} becomes
$$
\Phi(z)=\sum_{n=1}^\infty (-1)^{Nn-\frac {n-1} {p-1}}
\frac {1} {n}\binom {Nn} {\frac {n-1} {p-1}}\,z^n.
$$
A straightforward application of the Lagrange inversion formula
(cf.\ \cite[Theorem~5.4.2]{StanBI}) shows that
$\Phi(z)$ is the unique formal power series solution of the 
equation
\begin{equation} \label{eq:PhiGl}
\Phi(z)-z\left(\Phi^{p-1}(z)-1\right)^N=0.
\end{equation}

\begin{lemma} \label{lem:Phi'}
We have
\begin{equation} \label{eq:Phi'z} 
\Phi'(z)=\frac {\text{\em Pol}(z,\Phi(z))} 
{(-1)^{(p-1)N}\big((p-1)N\big)^{(p-1)N}z^{p-1}+
\big((p-1)N-1\big)^{(p-1)N-1}},
\end{equation}
where $\text{\em Pol}(z,t)$ is a polynomial in $z$ and $t$ over the integers.
\end{lemma}

\begin{proof}
Differentiating both sides of \eqref{eq:PhiGl}, we obtain
$$
\Phi'(z)-\left(\Phi^{p-1}(z)-1\right)^N
-zN(p-1)\Phi'(z)\Phi^{p-2}(z)\left(\Phi^{p-1}(z)-1\right)^{N-1}=0.
$$
Hence,
\begin{equation} \label{eq:Phi'} 
\Phi'(z)=\frac {\left(\Phi^{p-1}(z)-1\right)^N}
{1-zN(p-1)\Phi^{p-2}(z)\left(\Phi^{p-1}(z)-1\right)^{N-1}}.
\end{equation}
We must now express the reciprocal of the denominator
as a polynomial in $\Phi(z)$. In order to do this, we make
the Ansatz
\begin{equation} \label{eq:Glsys} 
\left(1-zN(p-1)\Phi^{p-2}(z)\left(\Phi^{p-1}(z)-1\right)^{N-1}\right)
\sum_{i=0}^{(p-1)N-1}b_i(z)\Phi^i(z)=1,
\end{equation}
with at this point undetermined coefficients $b_i(z)$, where
the sum represents the reciprocal of the denominator in \eqref{eq:Phi'}.
We multiply both sides of the last equation by
$\left(\Phi^{p-1}(z)-1\right)$. Then, using \eqref{eq:PhiGl}, we obtain
\begin{equation} \label{eq:bi} 
\left(\big(1-N(p-1)\big)\Phi^{p-1}(z)-1\right)
\sum_{i=0}^{(p-1)N-1}b_i(z)\Phi^i(z)=\Phi^{p-1}(z)-1,
\end{equation}
We expand the product on the left-hand side and use \eqref{eq:PhiGl}
again to reduce $\Phi^{(p-1)N}(z)$ to a linear combination of
lower powers of $\Phi(z)$. This leads to
\begin{multline*}
\big(1-N(p-1)\big)\sum_{i=0}^{(p-1)N-p}b_i(z)\Phi^{i+p-1}(z)\\
+
\big(1-N(p-1)\big)\sum_{i=0}^{p-2}b_{i+(p-1)(N-1)}(z)
\bigg(z^{-1}\Phi^{i+1}(z)-\sum_{k=0}^{N-1}\binom
Nk(-1)^{N-k}\Phi^{i+(p-1)k}(z)\bigg)\\
-\sum_{i=0}^{(p-1)N-1}b_i(z)\Phi^i(z)
=\Phi^{p-1}(z)-1,
\end{multline*}
Comparison of powers of $\Phi(z)$ then yields a system of equations
of the form
\begin{equation} \label{eq:Abc} 
M\cdot b=c,
\end{equation}
where $b=(b_i(z))_{0\le i\le (p-1)N-1}$ is the column vector of unknowns,
$c=(c_i)_{0\le i\le (p-1)N-1}$ 
with $c_0=-1$, $c_{p-1}=1$, and $c_i=0$ otherwise,
and $M$ is the $(p-1)N\times (p-1)N$ matrix given by
$$
M_{i,j}=\begin{cases} 
-1&\text{if }0\le i=j\le (p-1)N-p,\\
X&\text{if }p-1\le i=j+p-1\le (p-1)N-1,\\
Xz^{-1}&\text{if }1\le i=j-(p-1)(N-1)+1\le p-1,\\
(-1)^{N-k-1}\binom NkX&\text{if }0\le i-(p-1)k=j-(p-1)(N-1)\le p-2,\\
&\qquad \qquad \text{for some $k$ with }0\le k\le N-1,
\end{cases}
$$
$X$ being short for $1-(p-1)N$. The structure of the matrix $M$
becomes clearer if we reorder the rows and columns of the matrix
simultaneously so that first come the rows and columns indexed
by $i$ and $j$ which are $\equiv0$~(mod~$p-1$), respectively,
then those which are $\equiv1$~(mod~$p-1$), \dots, and finally
those which are $\equiv p-2$~(mod~$p-1$). The result is the
matrix
\begin{equation} \label{eq:M} 
\begin{pmatrix} 
A&0&0&\dots&0&C\\
B&A&0&\dots&0&0\\
0&B&A&\dots&0&0\\
\vdots&&\ddots&\ddots&&\vdots\\
0&\dots&0&B&A&0\\
0&\dots&0&0&B&A
\end{pmatrix},
\end{equation}
where the block $A$ is the $N\times N$ matrix given by
\begin{equation} \label{eq:A} 
A=
\begin{pmatrix} 
-1&0&0&\dots&0&(-1)^{N-1}X\\
X&-1&0&\dots&0&(-1)^{N-2}X\binom N1\\
0&X&-1&\dots&0&(-1)^{N-3}X\binom N2\\
\vdots&&\ddots&\ddots&&\vdots\\
0&\dots&0&X&-1&-X\binom N{N-2}\\
0&\dots&&0&X&X\binom N{N-1}-1\\
\end{pmatrix},
\end{equation}
$B$ is the $N\times N$ matrix given by
$$B=
\begin{pmatrix} 
0&\dots&0&Xz^{-1}\\
0&\dots&0&0\\
\vdots&&\vdots&\vdots\\
0&\dots&0&0\\
\end{pmatrix},
$$
and $C$ is the $N\times N$ matrix given by
$$C=
\begin{pmatrix} 
0&\dots&0&0\\
0&\dots&0&Xz^{-1}\\
0&\dots&0&0\\
\vdots&&\vdots&\vdots\\
0&\dots&0&0\\
\end{pmatrix}.
$$
The determinant of $M$ (that is, of the matrix in \eqref{eq:M}) 
is computed in Lemma~\ref{lem:M}.
It is obviously non-zero, therefore the system of linear
equations satisfied by the coefficients $b_i(z)$, $i=0,1,\dots,N-1$,
has a unique solution. In the end, we obtain \eqref{eq:Phi'z}.
\end{proof}

\begin{lemma} \label{lem:M}
The determinant of the matrix in \eqref{eq:M} equals
$$
(-1)^{(p-1)N}\big((p-1)N\big)^{(p-1)N}+
\big((p-1)N-1\big)^{(p-1)N-1}z^{-p+1}.$$
\end{lemma}

\begin{proof}
We write the last column of \eqref{eq:M} 
as the sum $c_1+c_2$, where $c_1$ is
the column with $Xz^{-1}$ as index-1 entry 
(the reader should remember that our
indexing starts with~$0$) and $0$'s otherwise, and $c_2$ 
is the ``rest," that is, the index~$i$ entry equals
$$(-1)^{N-(i-(p-2)N)-1)}\binom N{i-(p-2)N}X-\delta_{i,(p-1)N-1}$$ 
for $i=(p-2)N,(p-2)N+1,
\dots,(p-1)N-1$ and $0$'s otherwise.
Then, by linearity in the last column, the determinant $\det M$
equals the sum of 
\begin{equation} \label{eq:M1} 
\det\begin{pmatrix} 
A&0&0&\dots&0&0\\
B&A&0&\dots&0&0\\
0&B&A&\dots&0&0\\
\vdots&&\ddots&\ddots&&\vdots\\
0&\dots&0&B&A&0\\
0&\dots&0&0&B&A
\end{pmatrix},
\end{equation}
and the determinant of a second matrix, which arises from \eqref{eq:M}
by replacing the last column by $c_1$.
Since the matrix in \eqref{eq:M1} is a lower triangular block
matrix, its determinant equals
\begin{equation} \label{eq:detM1a}
\left(\det A\right)^{p-1}.
\end{equation}
We are going to evaluate the determinant of~$A$ in Lemma~\ref{lem:detA}.

In order to evaluate the determinant of the second matrix, we expand
it along the last column. This leads to the expression
\begin{equation} \label{eq:M2} 
(-1)^{(p-1)N}Xz^{-1}
\det\begin{pmatrix} 
A'&0&0&\dots&0&0\\
B&A&0&\dots&0&0\\
0&B&A&\dots&0&0\\
\vdots&&\ddots&\ddots&&\vdots\\
0&\dots&0&B&A&0\\
0&\dots&0&0&B&A''
\end{pmatrix},
\end{equation}
where $A'$ is the matrix which arises from $A$ by deleting
its row with index~$1$ (it should be remembered again that our
indexing of rows starts with the index~$0$), and $A''$ is the
matrix which arises from $A$ be deleting its last column.
Inspection of the matrix in \eqref{eq:M2} reveals that it is
an upper (sic!) triangular matrix, hence its determinant equals the
product of its diagonal entries, so that \eqref{eq:M2} equals
\begin{equation*}
(-1)^{(p-1)N}Xz^{-1}\left(Xz^{-1}\right)^{p-2}(-1)X^{(p-1)(N-1)-1}
=
\big((p-1)N-1\big)^{(p-1)N-1}z^{-p+1}.
\hfill\qedhere
\end{equation*}
\end{proof}

\begin{lemma} \label{lem:detA}
With the matrix $A$ given by \eqref{eq:A}, We have
$$
\det A=(-1)^{N}\big((p-1)N\big)^N.
$$
\end{lemma}

\begin{proof}
We replace the $0$-th row of $A$ by
$$
\sum_{j=0}^{N-1}X^{-j}\cdot\text{(row $j$)}.
$$
This is an operation which does not change the determinant.
For the entry in the $0$-th row and $(N-1)$-st column of the
new matrix, we obtain
\begin{align}
\notag
\sum_{j=0}^{N-1}X^{-j}(-1)^{N-j-1}\binom NjX-X^{-N+1}
&=-X^{-N+1}\sum_{j=0}^{N}X^{N-j}(-1)^{N-j}\binom Nj\\
&=-X^{-N+1}(1-X)^N=-\frac {\big((p-1)N\big)^N} {X^{N-1}}.
\label{eq:Y}
\end{align}
Thus, after the operation described above, the new matrix reads
$$
\begin{pmatrix} 
0&0&0&\dots&0&Y\\
X&-1&0&\dots&0&(-1)^{N-2}X\binom N1\\
0&X&-1&\dots&0&(-1)^{N-3}X\binom N2\\
\vdots&&\ddots&\ddots&&\vdots\\
0&\dots&0&X&-1&-X\binom N{N-2}\\
0&\dots&&0&X&X\binom N{N-1}-1\\
\end{pmatrix},
$$
where $Y$ denotes the quantity in \eqref{eq:Y}.
The determinant of this matrix, and thus the determinant of~$A$, equals
$$
(-1)^{N-1}X^{N-1}Y=(-1)^{N}\big((p-1)N\big)^N,
$$
establishing the claim.
\end{proof}

\section{A generating function approach}
\label{sec:GF}

Given a finitely generated virtually free group $\Gamma$,
in this section we write $F(z)$ for the generating function
$F_\Gamma(z)=\sum_{\lambda=1}^\infty f_\lambda(\Gamma)\,z^\lambda$ 
of the number $f_\lambda(\Gamma)$ of subgroups of index
$m_\Gamma\lambda$ in~$\Gamma$. The theorem below shows that,
under the conditions of Theorem~\ref{thm:mup=0}, the series
$F(z)$, when coefficients are reduced 
modulo any given $p$-power,
can be expressed as a polynomial in $\Phi(z)$.

\begin{theorem} \label{thm:1}
Let $p$ be a prime and $\alpha$ a positive integer. Furthermore, let
$\Gamma$ be a finitely generated virtually free group with
$\mu_p(\Gamma)=0$ and $\mu(\Gamma)\ge2$. As before, let
$\Phi(z)$ be the series in \eqref{eq:Phi}.
Then the generating function 
$F(z)$ for the free subgroup numbers
of\/~$\Gamma$, when reduced modulo $p^\alpha,$ 
can be expressed as a polynomial in $\Phi(z)$ of degree at most
$\mu(\Gamma)-1,$ with coefficients in
$\Z[z,z^{-1},Y^{-1}(z)]$, where
$$
Y(z)=\begin{cases} 
z^{p-1}-\big(\frac {\mu(\Gamma)} {p-1}+1\big)^{-1},&
\text{if }p\ge3\text{ and }\mu(\Gamma)\not\equiv0,1~(\text{\em
  mod}~p),\\
1,&\text{otherwise}.
\end{cases}
$$
\end{theorem}

\begin{proof}
It is known from \cite[Prop.~2]{MuPuFree} 
that $F(z)$ satisfies a differential
equation of the form
\begin{equation} \label{eq:Xeq} 
F(z)=z\left(F^{p-1}(z)-1\right)^{\mu(\Gamma)/(p-1)}
+p\mathcal P\left(z,F(z),F'(z),\dots,F^{(\mu(\Gamma)-1)}(z)\right),
\end{equation}
where $\mathcal P(z,t_0,t_1,\dots,t_{\mu(\Gamma)-1})$ is a polynomial
in $z,t_0,t_1,\dots,t_{\mu(\Gamma)-1}$ over the integers.\break
(To be precise, this is the result of a careful $p$-adic analysis of
the differential equation arising from 
a combination of \eqref{Eq:GFTransform} and \eqref{Eq:G(z)Diff}.)
It is our goal
to express $F(z)$ modulo~$p^\al$ as a polynomial in $\Phi(z)$
with coefficients in~$\Z[z,z^{-1},Y^{-1}(z)]$. Since $\Phi(z)$
satisfies the functional equation \eqref{eq:PhiGl} with\break
$N=\mu(\Gamma)/(p-1)$, we have
$$
F(z)=\Phi(z)\quad \text{modulo }p.
$$
Here, given integral power series (or Laurent series) $f(z)$ and $g(z)$,
we write 
$$f(z)=g(z)~\text {modulo}~p^\ga$$ 
to mean that the coefficients
of $z^i$ in $f(z)$ and $g(z)$ agree modulo~$p^\ga$ for all $i$.

We now suppose that we have already found a polynomial
$$
F_\be(z)=\sum_{i=0}^{\mu(\Gamma)-1}a_{i,\be}(z)\Phi^i(z),
$$
with coefficients $a_{i,\be}(z)$ in $\Z[z,z^{-1},Y^{-1}(z)]$, so that
\begin{equation} \label{eq:Xbe} 
F(z)=F_\be(z)\quad \text{modulo }p^\be.
\end{equation}
We then make the Ansatz
\begin{equation} \label{eq:Ansatz} 
F(z)=F_{\be+1}(z)=F_\be(z)+p^\be\sum_{i=0}^{\mu(\Gamma)-1}
b_{i,\be+1}(z)\Phi^i(z)
\quad \quad 
\text{modulo }p^{\be+1},
\end{equation}
for certain, at this point undetermined, rational functions
$b_{i,\be+1}(z)$ over the integers. We substitute this Ansatz in the
differential equation \eqref{eq:Xeq} and reduce the
result modulo~$p^{\be+1}$, to obtain
\begin{multline} \label{eq:Fbe}
F_\be(z)+p^\be\sum_{i=0}^{\mu(\Gamma)-1}b_{i,\be+1}(z)\Phi^i(z)\\
-z\left(F_\be^{p-1}(z)
+(p-1)F_\be^{p-2}(z)p^\be\sum_{i=0}^{\mu(\Gamma)-1}
b_{i,\be+1}(z)\Phi^i(z)-1\right)^{\mu(\Gamma)/(p-1)}\\
-p\mathcal P\left(z,F_{\be+1}(z),F_{\be+1}'(z),\dots,
F_{\be+1}^{(\mu(\Gamma)-1)}(z)\right)=0
\quad \quad 
\text{modulo }p^{\be+1}.
\end{multline}
By Lemma~\ref{lem:Phi'}, we have
\begin{equation} \label{eq:Fj} 
F^{(j)}_{\be+1}(z)=F^{(j)}_\be(z)
+p^\be\sum_{i=0}^{\mu(\Gamma)-1}c_{i,\be+1;j}(z)\Phi^i(z)
\end{equation}
for all non-negative integers~$j$ and certain rational functions
$c_{i,\be+1;j}(z)$. It should be noted that the denominators of these
rational functions are powers of
\begin{equation} \label{eq:Nenner} 
(-1)^{(p-1)N}\big((p-1)N\big)^{(p-1)N}z^{p-1}+
\big((p-1)N-1\big)^{(p-1)N-1},
\end{equation}
where we wrote again $N=\mu(\Gamma)/(p-1)$ for short.
Since we are considering \eqref{eq:Fbe} modulo $p^{\beta+1}$,
and since the sum on the right-hand side of \eqref{eq:Fj} has
the prefactor $p^\be$, we may reduce these denominators modulo~$p$.
Explicitly, we have
\begin{multline} \label{eq:Nennerp}
(-1)^{(p-1)N}\big((p-1)N\big)^{(p-1)N}z^{p-1}+
\big((p-1)N-1\big)^{(p-1)N-1}\\
=\begin{cases} 
1\text{ (mod }p),&\text{if }p=2\text{ and }N\text{ is even,}\\
z\text{ (mod }p),&\text{if }p=2\text{ and }N\text{ is odd,}\\
-1\text{ (mod }p),&\text{if }p\ge3\text{ and }N\equiv0~\text{mod}~p\\
z^{p-1}\text{ (mod }p),
&\text{if }p\ge3\text{ and }N\equiv-1~\text{mod}~p\\
z^{p-1}-(N+1)^{-1}\text{ (mod }p),
&\text{if }p\ge3\text{ and }N\not\equiv0,1~\text{mod}~p\\
\end{cases}
\end{multline}
In all cases, the reciprocals of the polynomials on the right-hand
side of \eqref{eq:Nennerp} 
are elements of $\Z[z,z^{-1},Y^{-1}(z)]$.
(Here we use that $N\equiv-1$~(mod~$p$) and
$\mu(\Gamma)\equiv1$~(mod~$p$) are equivalent.) Hence, in our
computation, the coefficients
$c_{i,\be+1;j}(z)$ may be assumed to lie in $\Z[z,z^{-1},Y^{-1}(z)]$.

If relation \eqref{eq:Fj} is substituted in
\eqref{eq:Fbe}, then one sees that this congruence reduces to
\begin{multline*} 
F_\be(z)+p^\be\sum_{i=0}^{\mu(\Gamma)-1}b_{i,\be+1}(z)\Phi^i(z)\\
-z\left((F_\be^{p-1}(z)-1)^N
+p^\be(p-1)NF_\be^{p-2}(z)(F_\be^{p-1}(z)-1)^{N-1}
\sum_{i=0}^{\mu(\Gamma)-1}b_{i,\be+1}(z)\Phi^i(z)\right)\\
-p\mathcal P\left(z,F_{\be}(z),F_{\be}'(z),\dots,
F_{\be}^{(\mu(\Gamma)-1)}(z)\right)=0
\quad \quad 
\text{modulo }p^{\be+1}.
\end{multline*}
By definition of $F_\be(z)$, we may divide both sides by $p^\be$.
This leads to the congruence
\begin{multline*} 
G_\be(z)+\sum_{i=0}^{\mu(\Gamma)-1}b_{i,\be+1}(z)\Phi^i(z)\\
-(p-1)NzF_\be^{p-2}(z)(F_\be^{p-1}(z)-1)^{N-1}
\sum_{i=0}^{\mu(\Gamma)-1}b_{i,\be+1}(z)\Phi^i(z)
=0
\quad \quad 
\text{modulo }p,
\end{multline*}
for some explicitly given polynomial $G_\be(z)$ in $\Phi(z)$ with
coefficients in $\Z[z,z^{-1},Y^{-1}(z)]$.
By construction, we have
$$
F_{\be}(z)=\Phi(z)\quad \text{modulo }p.
$$
Using this in the above congruence, we arrive at
\begin{multline} \label{eq:bcong}
G_\be(z)+\left(1
-(p-1)Nz\Phi^{p-2}(z)(\Phi^{p-1}(z)-1)^{N-1}\right)
\sum_{i=0}^{\mu(\Gamma)-1}b_{i,\be+1}(z)\Phi^i(z)
=0\\
\quad \quad 
\text{modulo }p.
\end{multline}
By reducing ``high" powers of $\Phi(z)$ by means of \eqref{eq:PhiGl}
and subsequently 
comparing coefficients of powers of $\Phi(z)$, we obtain a system
of linear equations over $\Z/p\Z$ for the unknown rational functions
$b_{i,\be+1}(z)$, $i=0,1,\dots,\mu(\Gamma)-1$. As inspection shows, the
coefficient matrix of the system is exactly the same as the one
arising from \eqref{eq:bi}. 
(The reader should in particular compare \eqref{eq:bcong} and
\eqref{eq:Glsys}.)
We have computed the determinant
of this coefficient matrix in the proof of
Lemma~\ref{lem:Phi'}. As a matter of fact, it
equals
\eqref{eq:Nenner} divided by $z^{p-1}$. 
Since \eqref{eq:bcong} is a congruence modulo~$p$, we have to reduce
\eqref{eq:Nenner} modulo~$p$, which we did in \eqref{eq:Nennerp}.
We observed that the reciprocals of the reduced expressions lie in
$\Z[z,z^{-1},Y^{-1}(z)]$ in all cases. In particular, they are all non-zero.
Hence, there are
unique rational functions $b_{i,\be+1}(z)$, $i=0,1,\dots,\mu(\Gamma)-1$,
solving \eqref{eq:bcong}, and all of them are elements of
$\Z[z,z^{-1},Y^{-1}(z)]$. This completes the proof of the theorem.
\end{proof}

\section{The main results}
\label{sec:main}

Let again $\Gamma$ be a finitely generated virtually free group and
$p$ a prime such that $\mu_p(\Gamma)=0$ and
$\mu(\Gamma)\ge2$.
We are now in the position to derive the main results of this paper,
which say that the number of free subgroups 
of index $m_\Gamma\lambda$ in $\Gamma$,
when reduced modulo any given $p$-power, 
is congruent to a binomial coefficient involving~$\lambda$ times
a rational function in $\lambda$, respectively a sum involving these
quantities. These results are made precise in
Corollaries~\ref{cor:1} and \ref{cor:2} below. We accompany these
results by concrete examples, given in Example~\ref{ex:1} and
\ref{ex:2}, which illustrate Corollary~\ref{cor:1}, respectively
Example~\ref{ex:3}, illustrating Corollary~\ref{cor:2}.
Moreover, we explain in Remarks~\ref{rem:1} and \ref{rem:2} how
the earlier results in \cite{KKM,KM} fit into the more general
picture that we present here.

\begin{corollary} \label{cor:1}
Let $r$ be a non-negative integer.
With the assumptions of Theorem~{\em\ref{thm:1}}, if
$\mu(\Gamma)\equiv0,1$~{\em(mod}~$p${\em)}, then
\begin{equation} \label{eq:cong1} 
f_\lambda(\Gamma)
\equiv R_{\Gamma,p,r}(\lambda)\binom {\frac {\mu(\Gamma)\lambda}{p-1}} 
{\frac {\lambda-r}{p-1}}\pmod {p^\alpha},
\quad 
\text{for $\lambda\equiv r$~{\em(mod}~$p-1${\em)}},
\end{equation}
where, $R_{\Gamma,p,r}(\lambda)$ is a rational function
in~$\lambda$.
\end{corollary} 

\begin{proof}
By Theorem~\ref{thm:1}, the generating function
$\sum_{\lambda=1}^\infty f_\lambda(\Gamma)\,z^\lambda$ equals
\begin{equation} \label{eq:GF} 
\sum_{i=0}^{\mu(\Gamma)-1} b_{i,\alpha}(z)\Phi^i(z)
\quad \quad \text{modulo }p^\alpha,
\end{equation}
and the coefficients $b_{i,\alpha}(z)$
are elements of $\Z[z,z^{-1},Y^{-1}(z)]$. According to the definition
of $Y(z)$, under our assumption $\mu(\Gamma)\equiv0,1$ we have
$Y(z)=1$. Consequently, the coefficients $b_{i,\alpha}(z)$ are
actually Laurent polynomials over the integers.

We must now extract the coefficient of $z^\lambda$ in \eqref{eq:GF}.
In order to do so, we appeal again to the Lagrange inversion formula
(cf.\ \cite[Theorem~5.4.2]{StanBI}), which shows that
\begin{equation} \label{eq:Phim} 
\coef{z^n}\Phi^m(z)=(-1)^{\frac {(\mu(\Gamma)-1)n+m} {p-1}}
\frac {m} {n}\binom {\frac {\mu(\Gamma)n} {p-1}}{\frac {n-m} {p-1}}.
\end{equation}
If this is used to extract the coefficient of $z^\lambda$ in
\eqref{eq:GF} for $\lambda\equiv r$~(mod~$p-1$), then one
arrives at the assertion \eqref{eq:cong1}.
\end{proof}

\begin{example} \label{ex:1}
We let $p=3$ and $\Gamma_1$ a finitely generated virtually free group with 
order graph given by the normalised tree in Figure~\ref{fig:1}.
In this situation, we have $m_{\Gamma_1}=6$, 
$\zeta_1(\Gamma_1)=2$, $\zeta_2(\Gamma_1)=4$,
$\zeta_3(\Gamma_1)=0$, and $\zeta_6(\Gamma_1)=-1$, and thus
$$
\mu_3(\Gamma_1)
=1+\varphi\!\left(\tfrac {6} {3}\right)\zeta_3(\Gamma_1)
+\varphi\!\left(\tfrac {6} {6}\right)\zeta_6(\Gamma_1)
=1+0+(-1)=0
$$
and
$$
\mu(\Gamma_1)
=1+\varphi\!\left(\tfrac {6} {1}\right)\zeta_1(\Gamma_1)
+\varphi\!\left(\tfrac {6} {2}\right)\zeta_2(\Gamma_1)
+\varphi\!\left(\tfrac {6} {3}\right)\zeta_3(\Gamma_1)
+\varphi\!\left(\tfrac {6} {6}\right)\zeta_6(\Gamma_1)
=1+4+8+0-1=12.
$$

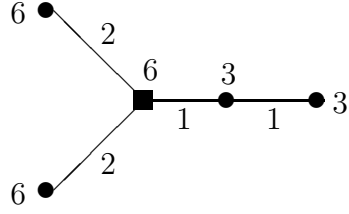
\begin{figure}
\unitlength.4cm
\begin{picture}(15,10)(-3,-1)
\put(4,5){\line(1,0){6}}
\put(4,5){\line(-1,1){3}}
\put(4,5){\line(-1,-1){3}}
\put(4,5){\raise0pt\hbox{\kern-4pt\large \vrule width7pt height3.5pt depth3.5pt}}
\put(1,8){\raise0pt\hbox{\kern-3pt\circle*{.5}}}
\put(1,2){\raise0pt\hbox{\kern-3pt\circle*{.5}}}
\put(7,5){\raise0pt\hbox{\kern-3pt\circle*{.5}}}
\put(10,5){\raise0pt\hbox{\kern-3pt\circle*{.5}}}
\put(3,3){\raise-5pt\hbox{\kern-5pt 2}}
\put(3,7.3){\raise-5pt\hbox{\kern-5pt 2}}
\put(5.5,4.5){\raise-5pt\hbox{\kern-5pt 1}}
\put(8.5,4.5){\raise-5pt\hbox{\kern-5pt 1}}
\put(10.7,5){\raise-5pt\hbox{\kern-5pt 3}}
\put(7,6){\raise-5pt\hbox{\kern-5pt 3}}
\put(4.4,6.1){\raise-5pt\hbox{\kern-5pt 6}}
\put(0,8){\raise-5pt\hbox{\kern-5pt 6}}
\put(0,2){\raise-5pt\hbox{\kern-5pt 6}}
\end{picture}
\vskip-.5cm
\caption{The order graph of the group $\Gamma_1$ in Example~\ref{ex:1}}
\label{fig:1}
\end{figure}
The functional equation for $F(z)=F_{\Gamma_1}(z)$ that we get from
\eqref{Eq:GFTransform} and \eqref{Eq:G(z)Diff}, after reduction of
the coefficients modulo~81, is
\begin{multline*}
63 z^3 F^9(z) F''(z)+72 z^3 F^3(z) F''(z)+27 z^3 F(z) F''(z)+72 z^2 
F^{10}(z) F'(z)\\
+36 z^2 F^9(z) F'(z)+9 z^2 F^4(z) F'(z)+18 z^2 F^3(z)
F'(z)+27 z^2 F^2(z) F'(z)\\
+27 z^2 F(z) F'(z)+54 z^2 F'(z)+z F^{12}(z) +36 z 
F^{11}(z) +15 z F^{10}(z) +72 z F^9(z) \\
+54 z F^8(z) +46 z F^6(z) +18 z
F^5(z) +21 
z F^4(z) +63 z F^3(z) \\
+9 z F^2(z) +54 z F(z)+80 F(z)+16 z=0
\quad \quad \text{modulo }81.
\end{multline*}
The algorithm given in the proof of Theorem~\ref{thm:1} to find a
solution to this congruence yields
\begin{multline*} 
F(z)=
15 z + (27 z+1) \Phi(z) + 69 z \Phi^2(z) + 9 z \Phi^3(z) + 42 z \Phi^4(z) + 
 27 z \Phi^5(z) \\
+ 39 z \Phi^6(z) + 27 z \Phi^7(z) + 66 z \Phi^8(z) + 
 72 z \Phi^9(z) + 12 z \Phi^{10}(z)
\quad \quad \text{modulo }81.
\end{multline*}
Coefficient extraction then yields
$$
f_{2L+1}(\Gamma_1)\equiv
(-1)^{L+1}
\frac {P_1(L)} {(12L+1)_6}
\binom {12L+6}L\pmod {81},
\quad \text{for }L\ge1,
$$
where
$$
P_1(L)=18(473007 L^5 + 969687 L^4 + 765456 L^3 + 308998 L^2 + 
     72732 L+9080),
$$
and
$$
f_{2L}(\Gamma_1)\equiv
(-1)^{L+1}
\frac {P_2(L)} {(12L-6)\, (11L-4)_4 }
\binom {12L-6}{L-1}\pmod {81},
\quad \text{for }L\ge1,
$$
where
$$
P_2(L)=324  (48 L^6 - 528 L^5 + 6079 L^2 + 9091 L^4- 10582 L^3  - 1874 L + 
    286),
$$
with the Pochhammer symbol 
$(\alpha)_m$ being defined by
$(\alpha)_m:=\alpha(\alpha+1)\cdots(\alpha+m-1)$ for 
$m\ge1$, and $(\alpha)_0:=1$.
\end{example}

\begin{example} \label{ex:2}
We let $p=2$ and $\Gamma_2$ a finitely generated virtually free group with 
order graph given by the normalised tree in Figure~\ref{fig:2}.
In this situation, we have $m_{\Gamma_2}=30$, 
$\zeta_1(\Gamma_2)=0$, 
$\zeta_2(\Gamma_2)=0$,
$\zeta_3(\Gamma_2)=3$, 
$\zeta_5(\Gamma_2)=1$, 
$\zeta_6(\Gamma_2)=0$, 
$\zeta_{10}(\Gamma_2)=0$, 
$\zeta_{15}(\Gamma_2)=5$, 
and $\zeta_{30}(\Gamma_2)=-1$, and thus
\begin{align*}
\mu_2(\Gamma_2)&=1
+\varphi\!\left(\tfrac {30} {2}\right)\zeta_2(\Gamma_2)
+\varphi\!\left(\tfrac {30} {6}\right)\zeta_6(\Gamma_2)
+\varphi\!\left(\tfrac {30} {10}\right)\zeta_{10}(\Gamma_2)
+\varphi\!\left(\tfrac {30} {30}\right)\zeta_{30}(\Gamma_2)\\
&=1+0+0+(-1)=0
\end{align*}
and
\begin{align*}
\mu(\Gamma_2)&=1
+\varphi\!\left(\tfrac {30} {1}\right)\zeta_1(\Gamma_2)
+\varphi\!\left(\tfrac {30} {2}\right)\zeta_2(\Gamma_2)
+\varphi\!\left(\tfrac {30} {3}\right)\zeta_3(\Gamma_2)
+\varphi\!\left(\tfrac {30} {5}\right)\zeta_5(\Gamma_2)
+\varphi\!\left(\tfrac {30} {6}\right)\zeta_6(\Gamma_2)\\
&\kern2cm
+\varphi\!\left(\tfrac {30} {10}\right)\zeta_{10}(\Gamma_2)
+\varphi\!\left(\tfrac {30} {15}\right)\zeta_{15}(\Gamma_2)
+\varphi\!\left(\tfrac {30} {30}\right)\zeta_{30}(\Gamma_2)\\
&=1+0+0+12+2+0+0+5+(-1)=19.
\end{align*}

\begin{figure}
\unitlength.4cm
\begin{picture}(15,10)(-3,-1)
\put(7,5){\line(-1,0){6}}
\put(4,5){\line(-1,1){3}}
\put(4,5){\line(-1,-1){3}}
\put(7,5){\line(1,0){3}}
\put(7,5){\raise0pt\hbox{\kern-4pt\large \vrule width7pt height3.5pt depth3.5pt}}
\put(1,8){\raise0pt\hbox{\kern-3pt\circle*{.5}}}
\put(1,2){\raise0pt\hbox{\kern-3pt\circle*{.5}}}
\put(10,5){\raise0pt\hbox{\kern-3pt\circle*{.5}}}
\put(4,5){\raise0pt\hbox{\kern-3pt\circle*{.5}}}
\put(1,5){\raise0pt\hbox{\kern-3pt\circle*{.5}}}
\put(3,3){\raise-5pt\hbox{\kern-5pt 3}}
\put(3,7.3){\raise-5pt\hbox{\kern-5pt 5}}
\put(5.5,4.5){\raise-5pt\hbox{\kern-5pt 15}}
\put(2.5,4.5){\raise-5pt\hbox{\kern-5pt 3}}
\put(8.5,4.5){\raise-5pt\hbox{\kern-5pt 3}}
\put(10.7,5){\raise-5pt\hbox{\kern-5pt 6}}
\put(7,6){\raise-5pt\hbox{\kern-5pt 30}}
\put(4.4,6.1){\raise-5pt\hbox{\kern-8pt 30}}
\put(0,8){\raise-5pt\hbox{\kern-10pt 10}}
\put(0,5){\raise-5pt\hbox{\kern-5pt 6}}
\put(0,2){\raise-5pt\hbox{\kern-5pt 6}}
\end{picture}
\vskip-.5cm
\caption{The order graph of the group $\Gamma_2$ in Example~\ref{ex:2}}
\label{fig:2}
\end{figure}
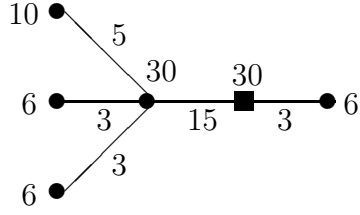
The functional equation for $F(z)=F_{\Gamma_2}(z)$ that we get from
\eqref{Eq:GFTransform} and \eqref{Eq:G(z)Diff}, after reduction of
the coefficients modulo~16, is
\begin{multline*}
4 z^3 F^{16}(z) F''(z)+8 z^3 F^8(z) F''(z)+4 z^3 F''(z)+10 z^2 F^{17}(z) 
F'(z)+10 z^2 F^{16}(z) F'(z)\\
+8 z^2 F^{13}(z) F'(z)+8 z^2 F^{12}(z) F'(z)+12 
z^2 F^9(z) F'(z)+12 z^2 F^8(z) F'(z)\\
+8 z^2 F^5(z) F'(z)+8 z^2 F^4(z) 
F'(z)+10 z^2 F(z) F'(z)+10 z^2 F'(z)+z F^{19}(z)\\
 +13 z F^{18}(z) +11 z 
F^{17}(z) +7 z F^{16}(z) +4 z F^{15}(z) +4 z F^{14}(z) +12 z F^{13}(z)\\
+12 z 
F^{12}(z) +14 z F^{11}(z) +6 z F^{10}(z) +10 z F^9(z) +2 z F^8(z) +4 z
F^7(z) +4 z 
F^6(z) \\+12 z F^5(z) +12 z F^4(z) +9 z F^3(z) +5 z F^2(z) +3 z F(z)+15 F(z)+15 z
=0
\quad \quad \text{modulo }16.
\end{multline*}
The algorithm given in the proof of Theorem~\ref{thm:1} to find a
solution to this congruence yields
\begin{multline*}
F(z)=4 z + 5 \Phi(z) + (12 z+2) \Phi^2(z) + 8 \Phi^3(z) + 12 \Phi^4(z) +
        8 \Phi^5(z) +  
 8 z \Phi^8(z) \\
+ 8 z \Phi^{10}(z) + 4 z \Phi^{16}(z) + 12 z \Phi^{18}(z)
\quad \quad \text{modulo }16.
\end{multline*}
Coefficient extraction then yields
\begin{multline*}
f_{\lambda}(\Gamma_2)\\
\equiv \frac {Q(\lambda)} 
{3 \lambda (6 \lambda+1) (9 \lambda+1) (9 \lambda+2) (18 \lambda+5) 
(19 \lambda-18)_{19} }
\binom {19\lambda}{\lambda-1}
\pmod{16},\\\text{for }\lambda\ge2,
\end{multline*}
where
\begin{multline*}
Q(\lambda)=41487381613117440000 + 1687131469740810240000 \lambda \\
+ 
     11694465019743123456000 \lambda^2 + 292824544319204134118400
     \lambda^3 \\
- 
     2920284679646876757433344 \lambda^4 + 29139678526675320716647104
     \lambda^5 \\
- 
     208744430518331785363075776 \lambda^6 + 
     1109655351908161743775529040 \lambda^7 \\
- 
     4529445293042933659974133664 \lambda^8 + 
     13823323659414730061860809764 \lambda^9 \\- 
     27457006500072077685531953836 \lambda^{10} + 
     13774006864417015570820956495 \lambda^{11} \\
+ 
     106285230034124606189268827556 \lambda^{12} - 
     297352958635465036740864629691 \lambda^{13} \\
- 
     141581261268484414672371284786 \lambda^{14} + 
     3042215815187103665497014434600 \lambda^{15}\\
 - 
     10200061275321550038724683325744 \lambda^{16} 
+ 
     20246947276823841509192253805174 \lambda^{17} \\
- 
     27403542237122957637017406285816 \lambda^{18} + 
     26128885491619758888717502991655 \lambda^{19} \\
- 
     17392298204833244937049876124804 \lambda^{20} + 
     7727636538613299232368005827649 \lambda^{21} \\
- 
     2065181275328822431645181305786 \lambda^{22} + 
     251508577253835734501825269810 \lambda^{23}.
\end{multline*}

\end{example}

\begin{remark} \label{rem:1}
More generally, if $p=2$, $\mu_2(\Gamma)=0$ and $\mu(\Gamma)\ge2$,
then we are always in the case covered by Corollary~\ref{cor:1},
since, trivially, $\mu(\Gamma)\equiv0$~(mod~$2$) or 
$\mu(\Gamma)\equiv1$~(mod~$2$). In particular, we see that
the discussion of the subgroup numbers of lifts of Hecke groups
$\mathfrak{H}(q) \cong C_2 \ast C_q$ with $q$ a Fermat
prime modulo powers of~$2$ in \cite[Sec.~8 and second part of Sec.~13]{KKM}
fits into the framework of Corollary~\ref{cor:1}, which can
be regarded as a vast generalisation. It has to be emphasised
yet that the results for lifts of Hecke groups in \cite{KKM} 
go slightly further than Corollary~\ref{cor:1} in that case
as the basic series used there --- which is
the mod-2-reduction of our basic series $\Phi(z)$ --- allows
for a very efficient coefficient extraction, a point that we
did not touch in the present paper.
\end{remark}

Now we turn to the somewhat more complicated case when
$\mu(\Gamma)\not\equiv0,1$~(mod~$p$). 
\begin{corollary} \label{cor:2}
Let $r$ be an integer with $0\le r\le p-2$.
With the assumptions of Theorem~{\em\ref{thm:1}}, if
$\mu(\Gamma)\not\equiv0,1$~{\em(mod}~$p${\em)}, then
\begin{multline} \label{eq:cong2} 
f_\lambda(\Gamma)
\equiv 
\left(\frac {\mu(\Gamma)} {p-1}+1\right)^{\lambda/(p-1)}
R^{(1)}_{\Gamma,p,r}(\lambda)
\\
+\sum_{k=0}^{\fl{\lambda/(p-1)}}
\left(\frac {\mu(\Gamma)} {p-1}+1\right)^kR^{(2)}_{\Gamma,p,r}(\lambda,k)
\binom {\frac {\mu(\Gamma)\lambda}{p-1}-\mu(\Gamma)k} 
{\frac {\lambda-r}{p-1}-k}\pmod {p^\alpha},
\\
\text{for $\lambda\equiv r$~{\em(mod}~$p-1${\em)}},
\end{multline}
where $R^{(1)}_{\Gamma,p,r}(\lambda)$ and
$R^{(2)}_{\Gamma,p,r}(\lambda,k)$ are rational functions in their
respective arguments.
Moreover, $R^{(2)}_{\Gamma,p,r}(\lambda,k)$ depends only on
$\frac {\lambda} {p-1}-k$.
\end{corollary} 

\begin{proof}
We begin as in the proof of Corollary~\ref{cor:2} by quoting
Theorem~\ref{thm:1}, which tells us that the generating function
$\sum_{\lambda=1}^\infty f_\lambda\,z^\lambda$ is given by
\eqref{eq:GF} modulo $p^\alpha$.
However, here we have $Y(z)=z^{p-1}-(N+1)^{-1}$, with
$N=\mu(\Gamma)/(p-1)$. 

Again, we must now extract the coefficient of $z^\lambda$ in \eqref{eq:GF}.
Here, we must first expand fractions,
$$
\frac {1} {Y^q(z)}=\frac {1} {\big(z^{p-1}-(N+1)^{-1}\big)^q}
=(-1)^{q}(N+1)^q\sum_{k=0}^\infty\binom {q+k-1}k(N+1)^kz^{(p-1)k}.
$$
Subsequent coefficient extraction using \eqref{eq:Phim} leads to
the result in \eqref{eq:cong2}, where the term containing
$R^{(1)}_{\Gamma,p,r}(\lambda)$ comes from the summand $b_{0,\alpha}(z)$ 
in \eqref{eq:GF}, while the term containing
$R^{(2)}_{\Gamma,p,r}(\lambda,k)$ is generated by the remaining
summands in \eqref{eq:GF}.
\end{proof}

We now show that the sum on the right-hand side of \eqref{eq:cong2}
satisfies a linear recurrence with constant coefficients, so
that the computation of this sum modulo~$p^\alpha$ can be
achieved in essentially linear time with growing $\lambda$
by reducing results modulo~$p^\alpha$ after each iteration
of the recurrence.
(We say ``essentially" since the computation of the inhomogeneous
part of the recurrence does grow super-linearly.)

\begin{proposition} \label{prop:rek}
Let $r$ be an integer with $0\le r\le p-2$, and
let $S_r(\lambda)$ denote the 
sum on the right-hand side of \eqref{eq:cong2}.
Then we have
\begin{equation} \label{eq:rek}
\sum_{j=0}^{d+1} (-1)^j\binom {d+1}j M^{d+1-j} S_r(\lambda+(p-1)j)
=g_r(\lambda),
\quad \text{for }\lambda\equiv r~\text{\em(mod~$p-1$)},
\end{equation}
where $d$ is the numerator degree of $R^{(2)}_{\Gamma,p,r}(\lambda,k)$ in
$\lambda$, $M=\frac {\mu(\Gamma)} {p-1}+1$,
and $g_r(\lambda)$ is a hypergeometric term, that is,
$g(\lambda+1)/g(\lambda)$ equals a rational function in~$\lambda$.
\end{proposition}

\begin{proof}
We fix $r$, and we write $\lambda=(p-1)L+r$. Using this notation,
the sum $S_r(\lambda)$ has the form
\begin{equation} \label{eq:FormS} 
S_r(\lambda)=S_r((p-1)L+r)=\sum_{k=0}^L M^k A(L-k,L)\,f(L-k),
\end{equation}
where $A(x,y)$ is a polynomial in $x$ and $y$ of degree~$d$ in~$y$,
and $f(L-k)$ comprises the binomial coefficient on the right-hand side
of \eqref{eq:cong2} as well as the denominator of 
$R^{(2)}_{\Gamma,p,r}(\lambda,k)$. We chose to parametrise the
polynomial $A(L-k,L)$ in this slightly unusual form since it
will be of advantage during the following computation.

Let $E$ denote the shift operator in $L$, that is, 
$(E h)(L):=h(L+1)$. We now apply
$(E-M\cdot\text{id})^{d+1}$ to \eqref{eq:FormS}. We obtain
\begin{align*}
(E-M&\cdot\text{id})^{d+1}S_r((p-1)L+r)\\
&=\sum_{j=0}^{d+1}(-1)^{d+1-j}\binom {d+1}jM^{d+1-j}\sum_{k=0}^{L+j} M^k
  A(L+j-k,L+j)\,f(L+j-k)\\ 
&=\sum_{j=0}^{d+1}(-1)^{d+1-j}\binom {d+1}jM^{d+1-j}\sum_{k=-j}^{L} M^{k+j}
  A(L-k,L+j)\,f(L-k)\\ 
&=\sum_{k=0}^{L} M^{d+k+1}f(L-k)
\sum_{j=0}^{d+1}(-1)^{d+1-j}\binom {d+1}j  A(L-k,L+j) + G_r(L),
\end{align*}
where $G_r(L)$ is the hypergeometric term resulting from summands
that we split off when passing from the summation over $k$ running
from $-j$ to $L$ in the next-to-last line to the summation over $k$
running from $0$ to $L$ in the last line.
Using the difference operator $\Delta_t$ defined by $(\Delta_t
h)(t):=h(t+1)-h(t)$, this last expression can be rewritten as
\begin{equation} \label{eq:Diff} 
(E-M\cdot\text{id})^{d+1}S_r((p-1)L+r)
=\sum_{k=0}^{L} M^{d+k+1}f(L-k)\,
\Delta_t^{d+1} A(L-k,t)\big\vert_{t=L} + G_r(L).
\end{equation}
We assumed that $A(x,y)$ is a polynomial of degree $d$ in $y$,
hence $\Delta_t^{d+1}$ kills $A(L-k,t)$. Consequently, on the right-hand side
in \eqref{eq:Diff} there remains only $G_r(L)$, while, after expansion
of $(E-M\cdot\text{id})^{d+1}$, the left-hand side becomes the
left-hand side of \eqref{eq:rek}, as desired.
\end{proof}

\begin{example} \label{ex:3}
We consider the Hecke group $\mathfrak H(7)=C_2*C_7$, whose order
graph is shown in Figure~\ref{fig:3}, and the prime $p=7$.
We have $m_{\mathfrak H(7)}=14$, 
$\zeta_1(\mathfrak H(7))=1$, 
$\zeta_2(\mathfrak H(7))=0$,
$\zeta_7(\mathfrak H(7))=0$, 
and $\zeta_{14}(\mathfrak H(7))=-1$, and thus
$\mu_7(\mathfrak H(7))=0$ and
$\mu(\mathfrak H(7))=6$.

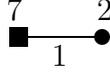
\begin{figure}
\unitlength.4cm
\begin{picture}(15,9)(0,-1)
\put(7,5){\line(1,0){3}}
\put(7,5){\raise0pt\hbox{\kern-4pt\large \vrule width7pt height3.5pt depth3.5pt}}
\put(10,5){\raise0pt\hbox{\kern-3pt\circle*{.5}}}
\put(10,6){\raise-5pt\hbox{\kern-5pt 2}}
\put(7,6){\raise-5pt\hbox{\kern-5pt 7}}
\put(8.5,4.5){\raise-5pt\hbox{\kern-5pt 1}}
\end{picture}
\vskip-1.5cm
\caption{The order graph of the group $\mathfrak H(7)$ in Example~\ref{ex:3}}
\label{fig:3}
\end{figure}
The functional equation for $F(z)=F_{\mathfrak H(7)}(z)$ that we get from
\eqref{Eq:GFTransform} and \eqref{Eq:G(z)Diff} is sufficiently small
to be displayed here:
{\allowdisplaybreaks
\begin{multline*}
537824 z^6 F^{(5)}(z)+230496 z^5 F(z) F^{(4)}(z)+6991712 z^5 
F^{(4)}(z)+41160 z^4 F(z)^2 F^{(3)}(z)\\
+1959216 z^4 F(z) F^{(3)}(z)+24989608 
z^4 F^{(3)}(z)+384160 z^5 \big(F''(z)\big)^2+3920 z^3 F^3(z)
F''(z)\\
+205800 z^3 F^2(z)  
F''(z)+3874528 z^3 F(z) F''(z)+25988424 z^3 F''(z)+41160 z^4 
\big(F'(z)\big)^3 \\
+8820 
z^3 F^2(z) \big(F'(z)\big)^2+288120 z^3 F(z) \big(F'(z)\big)^2
+2512132 z^3 \big(F'(z)\big)^2\\
+210 z^2 
F^4(z) F'(z)+9800 z^2 F^3(z) F'(z)+180516 z^2 F^2(z) F'(z)+1561336 z^2 F(z) 
F'(z)\\
+5336394 z^2 F'(z)+576240 z^5 F^{(3)}(z) F'(z)+164640 z^4 F(z) F'(z) 
F''(z)\\
+3649520 z^4 F'(z) F''(z)+z F^6(z) +42 z F^5(z) +679 z F^4(z) +5292 z 
F^3(z) \\
+20335 z F^2(z) +34986 z F(z)-F(z)+19305 z=0.
\end{multline*}}%
The algorithm in the proof of Theorem~\ref{thm:1} to find a
solution to this congruence gives
{\allowdisplaybreaks
\small
\begin{multline} \label{eq:H7GF}
\frac{\begin{matrix}7 \left(8 z^{18}+12 
z^{17}+7 z^{15}+7 z^{13}+48 z^{12}+16 z^{11}
\kern5.5cm
\right.\\\left.
\kern4cm
+7 z^{10}+35 z^9+42 z^7+23 
z^6+24 z^5+21 z^4+42 z^3+14 z\right)\end{matrix}}{\left(1-2 z^6\right)^3}\\
+\left(1+
\frac{\begin{matrix}7 \left(35 z^{18}+11 z^{17}+30 z^{16}+14 z^{14}+14 
z^{12}+17 z^{11}
\kern3cm
\right.\\\left.
\kern3cm
+47 z^{10}+14 z^9+42 z^8+21 z^6+z^5+32 z^4+42 
z^3\right)\end{matrix}}{\left(1-2 z^6\right)^3}\right) \Phi(z)\\
+\frac{\begin{matrix}7 \left(28 z^{17}+36 z^{16}+31 z^{15}+14 z^{14}+7 
z^{13}+21 z^{11}+20 z^{10}
\kern3cm
\right.\\\left.
\kern3cm
+18 z^9+35 z^8+28 z^7+7 z^5+30 z^4+41 z^3+28 
z^2+21 z\right)\end{matrix} }{\left(1-2 z^6\right)^3}\Phi^2(z)\\
+\frac{\begin{matrix}7 
\left(42 z^{16}+29 z^{15}+22 z^{14}+7 z^{13}+7 z^{12}+7 z^{10}+27 z^9
\kern3cm
\right.\\\left.
\kern3.5cm
+48 
z^8+35 z^7+21 z^6+35 z^4+16 z^3+2 z^2+42 z\right)\end{matrix} }{\left(1-2 
z^6\right)^3}\Phi^3(z)\\
+\frac{\begin{matrix}7 \left(42 z^{17}+14 
z^{15}+41 z^{14}+19 z^{13}+7 z^{12}+42 z^{11}
\kern4cm
\right.\\\left.
\kern3cm
+35 z^9+z^8+23 z^7+21 z^6+42 
z^5+28 z^3+26 z^2+31 z\right)\end{matrix} }{\left(1-2 z^6\right)^3}\Phi^4(z)\\
+\frac{\begin{matrix}7 \left(22 z^{18}+21 z^{17}+21 z^{16}+21 z^{14}+34 z^{12}+7 z^{11}
\kern3cm
\right.\\\left.
\kern4cm
+7 
z^{10}+28 z^8+9 z^6+28 z^5+28 z^4+42 z^2\right) \end{matrix}}{\left(1-2 
z^6\right)^3}\Phi^5(z)\\
\quad \quad \text{modulo }343.
\end{multline}}%
Finally, we have to extract coefficients. We content ourselves with
displaying here the results for $f_\lambda(\mathfrak H(7))$ for
$\lambda\equiv0$~(mod~6); for the other congruence classes
for~$\lambda$, similar results are available. By comparing
coefficients of $z^{6\lambda}$ on both sides of \eqref{eq:H7GF}, we obtain
\begin{multline} \label{eq:Res3}
f_{6\lambda}(\mathfrak H(7))\equiv
7\cdot2^{\lambda-2 } (  49 \lambda^2- 7 \lambda +4)\\
+
7\sum_{k=0}^\infty
(-1)^{k+\lambda} 2^{k-4} 
\frac{5\, (5 k-5 \lambda+1)\, (k-\lambda) S(\lambda) }
{3 \,(6\lambda-6k-5)_5}
\binom {6\lambda-6k}{\lambda-k}
\pmod{343},
\end{multline}
where
\begin{multline*}
S(\lambda)=22661 k^4-45322 k^3 
\lambda+70594 k^3+22661 k^2 \lambda^2-110545 k^2 \lambda+92331 k^2+39951 k \lambda^2
\\
-110913 k 
\lambda+56014 k+28424 \lambda^2-38696 \lambda+12528.
\end{multline*}
Let us denote the sum on the
right-hand side of the congruence \eqref{eq:Res3} by $S(\lambda)$.
Applying Proposition~\ref{prop:rek} (or, more precisely, its proof;
alternatively, one may use the Gosper--Zeilberger algorithm;
cf.\ \cite{PeWZAA}), 
we see that $S(\lambda)$ satisfies the recurrence
\begin{multline*}
S( \lambda+3)
- 6 S( \lambda+2)
+ 12 S( \lambda+1)
-8 S( \lambda)
\\
= 84\, (-1)^
       {l+1}\,T(\lambda)\frac {(l+1) (2 l-1) (3 l-2) (3 l-1) (6 l-5) ( 
         6 l-1) \, (6 l)! \,(6 l-6)!}
    {(5 l+11) ( 
         5 l+12) (5 l+13) \,(l+1 )! \,(5 l+10)!\, (6 l-1)!},
\end{multline*}
where
\begin{multline*}
T(\lambda)
=7578375074183 l^{12}
+ 110764942152696 l^{11} 
+ 719438896272607 l^{10} \\
+ 2739679993093800 l^9 
+ 6794561274739329 l^8 
+ 11525824255968648 l^7 \\
+ 13662933657289381 l^6 
+ 11354903297697240 l^5 
+ 6532000464773588 l^4 \\
+ 2520106018198656 l^3 
+   613697061412512 l^2 
+ 83672481893760 l 
+4738762828800.
\end{multline*}

\end{example}

\begin{remark} \label{rem:2}
The discussion of free subgroup numbers of lifts $\Gamma_m(3)$ 
of the classical
modular group $\mathfrak{H}(3)\cong \PSL_2(\Z)$ in \cite[Sec.~16]{KM}
taken modulo powers of~$3$
fits into the framework of Corollary~\ref{cor:2}. 
Indeed, for these lifts, we have 
$\mu_{\Gamma_m(3)}=2$, which is not congruent to $0,1$~(mod~$3$).
Consequently, according to Theorem~\ref{thm:1}, we must be prepared
to encounter denominators in the coefficients of the polynomial
in $\Phi(z)$ that expresses the generating function for the free
subgroup numbers when coefficients are reduced modulo a power of~$3$.
This is exactly what happened in \cite{KM}, and this is also the
reason why coefficient extraction was considerably harder in
\cite{KM} than in \cite{KKM}.
\end{remark}

\end{document}